\documentclass[11pt]{amsart}
\usepackage{amssymb,amsfonts,amscd,amstext}
\usepackage{epsfig,psfrag,graphicx}
\usepackage{graphicx,psfrag,mathrsfs} 
\usepackage{anysize}
 \usepackage{lineno} 
\newtheorem{theorem}{Theorem}[section]
\newtheorem{lemma}[theorem]{Lemma}
\newtheorem{corollary}[theorem]{Corollary}
\newtheorem{proposition}[theorem]{Proposition}

\newtheorem*{theoremA}{Theorem A}
\newtheorem*{theoremB}{Theorem B}
\newtheorem*{theoremC}{Theorem C}

\theoremstyle{definition}

\numberwithin{equation}{section}
\newcommand{\N}{{\mathbb{N}}}
\newcommand{\Z}{{\mathbb{Z}}}
\newcommand{\R}{{\mathbb{R}}}

\newcommand{\p}{\mathbf{p}}
\newcommand{\cqd}{\hfill $\square$}

\def\Z{\mathbb Z}

\def\1{{{\mathit 1} \!\!\>\!\! I} }
\renewcommand{\phi}{\varphi}

\title[Interval exchanges]{Orbit structure of   interval exchange \linebreak transformations with flip}


\medskip
\begin{document}

\maketitle

\centerline{\scshape Arnaldo Nogueira}
\smallskip
{\footnotesize
 \centerline{Institut de Math\'ematiques de Luminy, Universit\'e de la M\'editerran\'ee}
 \centerline {163, avenue de Luminy - Case 907, 13288 Marseille Cedex 9, France}
 \centerline{email: nogueira@iml.univ-mrs.fr}
} 

\medskip

\centerline{\scshape Benito Pires}

\smallskip

{\footnotesize
 \centerline{Departamento de Computa\c c\~ao e Matem\'atica da Universidade de S\~ao Paulo}
   \centerline{Av. Bandeirantes 3900, Monte Alegre, 14040-901, Ribeir\~ao Preto - SP, Brazil}
   \centerline{emails: benito@ffclrp.usp.br, bfpires@hotmail.com}
}

\medskip

\centerline{\scshape Serge Troubetzkoy}
\smallskip
{\footnotesize
\centerline{Centre de Physique Th\'eorique}
\centerline{F\'ed\'eration de Recherche des Unit\'es de Math\'ematiques de Marseille}
\centerline{Institut de Math\'ematiques de Luminy,
Universit\'e de la M\'editerran\'ee }
\centerline{163, avenue de Luminy - Case 907, 13288 Marseille Cedex 9, France}
\centerline{email: troubetz@iml.univ-mrs.fr}
}

\bigskip


\linespread{1.2}
\marginsize{2.5cm}{2.5cm}{1cm}{2cm}

\begin{abstract} A sharp bound on the number of invariant components of an interval exchange transformation is provided. 
More precisely, it is proved that the number of periodic components $n_{\rm per}$ and the number of minimal components 
$n_{\rm min}$ of an interval exchange transformation of $n$ intervals satisfy $n_{\rm per}+2\,n_{\rm min}\le n$. Besides, it 
is shown that
 almost all interval exchange transformations are typical, that is, have all the periodic components stable and all the minimal components 
robust (i.e. persistent under almost all small perturbations). Finally, we find all the possible values for
the integer vector $(n_{\rm per},n_{\rm min})$ for all typical interval exchange transformation of $n$ intervals.
\end{abstract}

\maketitle

\section{Introduction}

In this article we study interval exchange transformations (IETs) with flip(s). 
An IET is an injective piecewise isometry of an interval having finitely many jump discontinuities. An IET {\it has flip(s)} if it reverses the orientation of a subinterval of its domain, otherwise the IET is said to be {\it oriented}.
These maps are important objects in ergodic
theory, they have been intensively studied since the 1960s.  They occur naturally in the study of polygonal billiards, measured foliations
and flat surfaces \cite{MaTa,Masur}. In the context of polygonal billiards,  IETs were already studied, albeit in a different language,
in 1905 \cite{Lennes}. 
All the possible non-trivial recurrence of flows on closed surfaces can be explained by them \cite{Ar, Gu,Le}.  

The starting point of our investigation is the result of Nogueira \cite{N2} which states that almost every IET with flip(s) has (at least) one stable periodic trajectory. We want to know whether we can say more about the invariant components of these maps.

It is well known that IETs decompose into minimal and periodic components. This decomposition was first studied in the orientable (no flip(s)) case by Mayer in 1943 \cite{Ma} whereas in the case of IETs with flip(s), various partial or non-optimal versions have been demonstrated by other authors (\cite{Ara,KH,N2}).
We begin by giving in \mbox{Theorem A} a sharp estimate on the number of such components.  This result is obtained by means of a careful
analysis of the saddle-connections. 
In Theorem B, we show that for almost every IET with flip(s) all periodic components are stable and all minimal components are robust.
Here a  minimal component is called robust if it persists under almost all small perturbations.  The proof of this theorem is based on
Rauzy induction, it yields an almost sure ``algorithm'' for finding the decomposition into robust minimal components and stable periodic components.
 By using this algorithm, in Theorem C we give a complete classification of which decompositions can occur, this result turns out to be a converse of
Theorems A and B.

\section{Statement of the results}

An injective map $T:(0,c)\to (0,c)$ is an interval exchange transformation of $n$ intervals ($n$-IET) if there exist $n+1$
real numbers $0=x_0<x_1<\ldots<x_{n}=c$ such that $T\vert _{(x_{i-1},x_i)}$ is an isometry for all
$i\in\{1,\ldots,n\}$. The domain of $T$ is the set ${\rm Dom}\,(T)=\bigcup_{i=1}^n (x_{i-1},x_i)$. Notice that
the derivative $T'$ of $T$ is a locally constant function, moreover $\vert T'(x)\vert=1$ for all $x\in {\rm Dom}\,(T)$.
We say that $T$ is an {\it oriented} IET if  $T'(x)=1$ for all $x\in {\rm Dom}\,(T)$, otherwise  $T$ is called an IET {\it with flip(s)}. The interval $(x_{i-1},x_i)$ is  {\it flipped} if $T'(m_i)=-1$,
where $m_i=(x_{i-1}+x_i)/2$.

The  vector $\lambda=(\lambda_1,\ldots,\lambda_n)$ defined by $\lambda_i=x_i-x_{i-1}$ is called the {\it length vector}. We let $\theta=(\theta_1,\ldots,\theta_n)$ denote the {\it flip(s) vector} defined by
$\theta_i=T'(m_i)$, and $\pi$ be the permutation of $\{1,2,\ldots,n\}$ satisfying $T(m_{\pi^{-1}(1)})<T(m_{\pi^{-1}(2)})<\ldots < T(m_{\pi^{-1}(n)})$. In other words, $\pi(i)$ gives the position of the interval $T((x_{i-1},x_i))$ in the domain of $T^{-1}$. We may represent $\pi$ as an $n$-tuple by setting $\pi=(\pi_1,\ldots,\pi_n)=(\pi(1),\ldots,\pi(n))$. Finally, let $\p=(p_1,\ldots,p_n)$ be the {\it signed permutation} defined by $p_i=\theta_i \pi_i$.   A signed permutation $\p$ has {\it flip(s)} if $ p_i/\vert p_i\vert = -1$ for some $i$. 
Throughout the article we denote the permutation $\vert\p\vert=(\vert p_1\vert,\ldots,\vert p_n\vert)$ by the symbol $\pi$ and its {\em flip(s) vector} by $\theta = (\theta_1,\dots,\theta_n)$.

A signed permutation $\p$ is called {\it irreducible} if
$\pi (\{1,2,\ldots,k\})=\{1,2,\ldots,k\}$ implies $k=n$, otherwise we say that $\p$ is {\it reducible}. 
We denote by $\Lambda_n\subset\R^n$ the set of length vectors endowed with the Lebesgue measure (of the cone of positive vectors) and we let ${\mathscr{P}}_n$ (respectively ${\mathscr{P}}_n^{\rm irred}$, $\mathscr{P}_n^{\rm red}$) be the set of  signed permutations (respectively irreducible signed permutations, reducible signed permutations) endowed with the counting measure. The cartesian product $\Lambda_n\times{\mathscr{P}}_n$ is endowed with the product measure. In this way, to each $(\lambda,\p)\in \Lambda_n\times{\mathscr{P}}_n$ there corresponds an $n$-IET $T_{(\lambda,\p)}$. The IET $T_{(\lambda,\p)}$ is called  {\it irreducible} if $\p\in{\mathscr{P}}_n^{\rm irred}$. Given a subset $U\subset \Lambda_n\times {\mathscr P}_n$, we let
${\mathcal T}(U)=\{T_{(\lambda,\p)}\mid (\lambda,\p)\in U\}$ denote the set of $n$-IETs whose data belong to $U$.

An open interval $J\subset {\rm Dom}\,(T)$ is {\it rigid} if all positive
iterates $T^k$ of $T$ are defined (and so are continuous) on $J$.  A rigid interval
$J$ is a {\em maximal} rigid interval if any other rigid interval $K\subset {\rm Dom}\,(T)$
is either disjoint of $J$ or contained in $J$. If $J$ is a rigid interval then there exists $m\in\N$ such that
$T^m(J)=J$. We call the orbit $\bigcup_{k=0}^{m-1} T^k(J)$ a {\it periodic component} of $T$. 
Given two subsets $X,Y\subset\R$  and a point $y\in Y$, set $d(y,X)=\inf\{\vert y-x\vert:x\in X\}$. We also define
$\rho(X,Y)=\sup\, \{d(y,X):y\in Y\}. $
We say that a periodic component $O=\bigcup_{k=0}^{m-1} T^k(J)$  of $T$ is  {\it stable}   if for all $\epsilon>0$ there exists an open neighborhood $V_{\epsilon}$ 
of $T$ such that all $S\in V_{\epsilon}$ has a periodic component $O'$ satisfying $\rho(O,O')<\epsilon$.

Given an IET $T:(0,c)\to (0,c)$, let $\mathbb{I}=\mathbb{I}_+\cap\mathbb{I}_-$, where
\begin{eqnarray*}
 \mathbb{I}_+&=&\{x\in (0,c)\mid T^k(x)\in {\rm Dom}\,(T)\,\, \text{for all}\,\, k\in\N\},\\
 \mathbb{I}_-&=&\{x\in (0,c)\mid T^{-k}(x)\in {\rm Dom}\,\,(T^{-1})\,\, \text{for all}\,\, k\in\N\},
 \end{eqnarray*}
$T^0$ is the identity map and $T^k$ is the $k$-th iterate of $T$. The orbit of $x\in (0,c)$
is the set $O(x)=\{T^k(x)\mid k\in\Z\,\,\text{and}\,\, x\in{\rm Dom}\,(T^k)\}$. 
We call a non-empty open set $O \subset (0,c)$ a {\em
minimal component} of $T$ if $O={\rm int}\,(\overline{O(x)})$ for all $x\in O\cap(\mathbb{I}_-\cup\mathbb{I}_+$),
where ${\rm int}\,(B)$ (respectively $\overline{B}$) refers to the interior (respectively the topological closure) of the set $B$. In particular, in a minimal component $O$ every  infinite orbit starting at a point of $O$ is dense in $O$.
A minimal component $O$ of  $T$  is {\it robust} if for all $\epsilon>0$ there exist a  neighborhood $
V_\epsilon$ of $T$ and a measure zero set $N$ such that all $S\in V\setminus N$ have a minimal component $O'$ satisfying $\rho(O,O')<\epsilon$.

The main results of this paper are the following:

\begin{theoremA} The numbers $n_{\rm per}$ of periodic components and  $n_{\rm min}$ of minimal components of an $n$-IET
satisfy the inequality $n_{\rm per}+2\,{n_{\rm min}}\le n$. 
\end{theoremA} 

\begin{theoremB}\label{MT} Almost all interval exchange transformation have only stable periodic components and robust minimal components.
\end{theoremB}

The weaker upper bound $4n-4$ for the number of invariant components of all $n$-IET was proven in \cite[Theorem 14.5.13, p. 475]{KH}.  
For a full measure set of IETs with flip(s), \mbox{Aranson} proved in \cite[Theorem 3, p. 304]{Ara}  that $n_{\rm min}< n/2$ whereas \mbox{Nogueira \cite{N2}}  proved that the number of flipped periodic components belongs to the interval $[1,n]$. Here a periodic component \mbox{$O=\bigcup_{k=0}^{m-1} T^k(J)$} is {\it flipped} if $m$ is even and there exists $x\in J\cap{\rm Dom}\,(T^{\frac{m}{2}})$ such that $\big(T^{\frac{m}{2}}\big)'(x)=-1$.
Every flipped periodic component is stable but not all stable periodic components are flipped. For example, all \mbox{$2$-IET} with permutation $\p=(1,-2)$ have two stable periodic components but only one flipped periodic component. There are non-trivial examples of 4-IETs with irreducible permutation $\p=(-4,3,-2,-1)$ having four stable periodic components, only two of which are flipped.
In order to get the optimal upper bound of \mbox{Theorem A}, it is necessary to carefully control  the
counting arguments.






 
We conclude the article with the following existence result.
 
 \begin{theoremC}\label{222} Let $n\ge 1$ and $k,\ell\ge 0$ be integers. There exists an irreducible $n$-IET with flip(s) having $k$ stable periodic components and $\ell$ robust minimal components if and only if either of the following conditions are satisfied:
 $(k\ge 1, 1\le \ell<n/2\,\,{\rm and}\,\, k+2\,\ell\le n)$ or  $(k=n\,\,{\rm and}\,\,\ell=0)$. 
 \end{theoremC}
 
 \section{Counting invariant components}



In this section we will prove Theorem A. Let $T:(0,c)\to (0,c)$ be an $n$-IET and let
\mbox{$0=x_0<x_1<x_2<\ldots<{x_{n-1}}<x_n=c$} be the set of points where $T$ is not defined. We call
$\{x_1,\ldots,x_{n-1}\}$ the  {\it set of  singular points} of $T$ and we refer to $\{0,c\}=\partial [0,c]$ as the {\it set of endpoints} of $T$. 
We may assume that the singular points of $T$ are the discontinuities of $T$, otherwise $T$ could be considered as an IET of less 
intervals.
 We let $w_0^+=\lim_{x\to 0^+} T(x)$ and $w_n^-=\lim_{x\to x_n^-} T(x)$ be the one-sided limits of $T$ at its endpoints.
 Similarly, for all $j\in\{1,\ldots,n-1\}$,  set $w_j^+=\lim_{x\to x_j^+} T(x)$
 and $w_j^-=\lim_{x\to x_j^-} T(x)$.

By the above, we may
define the orbit of $0$ and of each singular point by continuity from the right.  
Similarly we define 
the orbit of $c$ and of each singular point by continuity from the left.  If one of these orbit 
continuations hits an endpoint  or a singular point we call it a {\em
saddle-connection}. In this way, $\gamma$ is a saddle-connection
if there exist $j,r\in\{0,1,\ldots,n\}$ and an integer $k\ge 0$ such that
$\gamma=\{x_j,T(w_j),\ldots,T^{k-1}(w_j),T^k(w_j)=x_r\}$, where 
$w_0=w_0^+$, $w_n=w_n^-$ and $w_j\in\{w_j^{+},w_j^{-}\}$ for
$1\le j\le n-1$. We always assume  $k\ge 0$ to be the smallest possible so that
$\gamma\cap\{x_i\}_{i=0}^n=\{x_j,x_r\}$. In this case, we say that $\gamma$ {\it starts} at $x_j$ and {\it ends} at
$x_r$.
It may happen that $x_j=x_r$ and
$\gamma=\{x_j\}$, in which case $k=0$ and $j=r$.

Accordingly with this definition, every IET   has at
least two saddle-connections, each of which ends at an endpoint $x\in\{0,c\}$ of $T$. These saddle-connections are called the {\it trivial saddle-connections} of $T$.

We say that a non-empty open set $O\subset (0,c)$ is a {\it transitive component} of $T$ if there exists $x\in\mathbb{I}$ such that \mbox{$O={\rm int}\,(\overline{O(x)})$}. 

\begin{lemma}\label{minvtran} $O$ is a minimal component if and only if $O$ is a transitive component.
\end{lemma}
\begin{proof} It is immediate that every minimal component is also a transitive component. The converse follows from \cite[Proposition 14.5.9, p. 473]{KH}.
\end{proof}

\begin{theorem}[Katok--Hasselblatt \cite{KH}]\label{KatHa} Let $T:(0,c)\to (0,c)$ be an $n$-IET. Then the following holds:
\item [(KH1)] There exist finitely many disjoint open sets $O_1$, \ldots, $O_N$, each of which consists of a finite union of (disjoint) open intervals, such that $[0,c]=\bigcup_{i=1}^N \overline{O_i}$. Each set $O_i$ is either a periodic component or a minimal component;
\item [(KH2)] If $O_i$ is a minimal component then $O_i$ is a finite union of (disjoint) open intervals with disjoint closures, in particular ${\rm int}\,(\overline{O_i})=O_i$;
\item [(KH3)]  $T(O_i\cap\mathbb{I})\subset O_i\cap\mathbb{I}$;
\item [(KH4)] The endpoints $\partial O_i$ of an invariant component $($periodic or minimal$)$ belong to saddle-connections.
\end{theorem}
\begin{proof}  The items (KH1) and (KH2) follow  from   \cite [Theorems 14.5.13 and 14.5.10]{KH} 
respectively, while (KH4) follows from  \cite [Lemma 14.5.4 and Corollary 14.5.9]{KH}. It remains to prove (KH3). If $O_i$ is periodic we have $T(O_i)=O_i$ and $O_i\cap
\mathbb{I}=O_i$, so in this case (KH3) is automatic. Now assume that $O_i$ is minimal and let
$x_0\in\mathbb{I}$ be such that $O(x_0)=\{T^k(x_0)\mid k\in\Z\}$ is dense in $O_i$. For each $x\in O_i\cap\mathbb{I}$, there exists a sequence $\{n_j\}_{j=0}^\infty$ with $\vert n_j\vert\to+\infty$ as $j\to\infty$ such that $x=\lim_{j\to\infty} T^{n_j}(x_0)$. In this way, $T(x)=\lim_{j\to\infty} T^{n_j+1}(x_0)
\subset \overline{O(x_0)}=\overline{O_i}$. Thus $T(O_i\cap\mathbb{I})\subset \overline{O_i}\cap \mathbb{I}$. Because $O_i$ is union of finitely many intervals, $\overline{O_i}=O_i\cup\partial O_i$. By (KH4),
$\partial O_i\cap \mathbb{I}=\emptyset$. Therefore, by the above, $T(O_i\cap\mathbb{I})\subset O_i\cap\mathbb{I}$.
\end{proof}


\begin{lemma}\label{3.2}  If $O_i$ is a minimal component then $O_i$ contains a singular point.
\end{lemma}
\begin{proof}  
Suppose that $O_i$ contains no singular points. Then
$O_i\cap {\rm Dom}\,(T)=O_i$. Because $T\vert_{{\rm Dom}\,(T)}$ is an open map, we have that
$T(O_i)$ is open and so ${\rm int}\,(T(O_i))=T(O_i)$. It follows from (KH3) and from the continuity of $T$ that $T(O_i)\subset \overline{O_i}$. Hence, by (KH2)
\begin{equation}\label{TOI}
T(O_i)={\rm int}\,(  T(O_i) )\subset {\rm int}\,(\overline{O_i})=O_i.
\end{equation}
Finally, by (KH1) and by (\ref{TOI}), there exist finitely many disjoint open intervals $I_1,\ldots, I_M$
and 
 a permutation $\alpha\in {\mathscr P}_M$ such that $O_i=\bigcup_{j=1}^M I_j$, 
  $T^k(I_1)=I_{\alpha(k)}$ for all $k\in\{1,\ldots,M\}$ and $T^M(I_1)=I_1$. Now, either $T^M$ or $T^{2M}$ is the identity map. In particular, $O_i$ is a periodic component, which contradicts the hypothesis. Therefore, $O_i$ contains a singular point.
\end{proof}

\begin{lemma}\label{3.3}  If $\overline{O_i}\cap \overline{O_j}\neq\emptyset$ for some $i\neq j$ then  $\partial {O_i}\cap \partial {O_j}$ contains a singular point.

\end{lemma}
\begin{proof} It follows from (KH3) and from the continuity of $T^k\vert_{{\rm Dom}\,(T^k)}$ that $T^k(\overline{O_i}\cap {\rm Dom}\,(T^k))\subset \overline{O_i}$ and $T^k(\overline{O_j}\cap {\rm Dom}\,(T^k))\subset \overline{O_j}$
for all $k\in\N$. In particular, $$T^k\big( \overline{O_i}\cap \overline{O_j}\cap {\rm Dom}\,(T^k)\big)\subset \overline{O_i}\cap\overline{O_j}\subset \partial {O_i}\cap\partial {O_j}.$$ By (KH4), each point in the boundary of an invariant component belongs to a saddle--connection. Hence, if $x\in \overline{O_i}\cap \overline{O_j}$, by the above there exists $k\in\N$  such that $T^k(x)$ is a singular point or an endpoint of $T$. By the above, $T^k(x)\in\partial {O_i}\cap\partial {O_j}$. It is easy to see that $T^k(x)$ cannot be an endpoint, so it has to be a singular point.
\end{proof}

\noindent{\it Proof of Theorem A}. Let $T:(0,c)\to (0,c)$ be an $n$-IET.    Let $O_1$, $O_2$, \ldots, $O_N$ be the invariant 
components of $T$. We denote
by $n_{\rm per}$ the number of periodic components and by $n_{\rm min}$ the number of minimal components. In this way, $n_{\rm per}+n_{\rm min}=N$. 

For each $1 \le i \le N$, let $y_i :=\inf O_i$. By (KH1), the numbers $y_1,\ldots,y_N$ are pairwise distinct, thus we can relabel the $O_i$'s so that $0=y_1<y_2<\ldots<y_N<c$.
We claim that there exists an injective map $\beta:\{1,2,\ldots,N\}\to\{0,1,\ldots,n-1\}$ that associates to each invariant component $O_{i}$, $i\ge 2$, a singular point $x_{\beta(i)}\in\partial O_{i}$ and satisfies $\beta(1)=0$. By the definition of $\{y_i\}_{i=1}^N$, 
 for all $i\in\{2,\ldots,N\}$ there exists
$1\le j<i$ such that $y_i\in\partial {O_i}\cap \partial {O_j}$. By Lemma \ref{3.3},   $\partial {O_i}\cap \partial {O_j}$ contains a singular point $x_{\beta(i)}$.
It remains to prove that $\beta$
is injective. Suppose that $\beta(i_1)=\beta(i_2)$ for some
$2\le i_1<i_2\le N$. Let $z=x_{\beta(i_1)}=x_{\beta(i_2)}$. By the above, there exist $1\le j_1<i_1$ and $1\le j_2<i_2$
such that $z\in \partial O_{i_1}\cap\partial O_{j_1}\cap \partial O_{i_2}\cap \partial O_{j_2}$. As $j_1<i_1<i_2$, the point $z$ belongs to the boundary of three different invariant components, which is not possible. Hence, $\beta$ is injective. This together with Lemma \ref{3.2} allows us to count the number of singular points necessary for having $N$ invariant components. There are two cases to consider (i) $O_{1}$ is a periodic component or (ii) $O_{1}$ is a minimal component. In case (i), we need two singular points for each minimal component and one singular point for each periodic component different from $O_{1}$. Hence, $2\, n_{\rm min}+n_{\rm per}-1\le n-1$, that is to say, $n_{\rm per}+2 \,n_{\rm min}\le n$. In case (ii), we need one singular point for each periodic component, two singular points for each minimal component different from $O_{1}$ and one singular point for $O_{1}$. Thus $2(n_{\rm min}-1)+1+n_{\rm per}\le n-1$, that is, $n_{\rm per}+2 \,n_{\rm min}\le n$.
\cqd

\section{Rauzy induction}



We shall denote by $\Lambda_n$ the subset of $\R^n$ formed by the length vectors
$$\Lambda_n=\{\lambda=(\lambda_1,\ldots,\lambda_n)\mid\lambda_i>0,\,\forall i\}.$$
For $\lambda\in\Lambda_n$  set
$$ \Vert \lambda  \Vert=\vert \lambda_1\vert+\vert\lambda_2\vert+\cdots+\vert \lambda_n\vert.$$
We say that $\lambda \in\Lambda_n$ is {\em irrational} if
the numbers $\lambda_1,\ldots,\lambda_n$ are rationally independent. We denote by $\Lambda_n^*$ the subset of $\Lambda_n$ formed by the  irrational length vectors. We say that an $n$-IET is {\it irrational} if its length vector is irrational.





We will denote by ${\mathscr P}_n^*$ the following class of signed permutations on $n$ symbols
$$
{\mathscr P}_n^*=\{\p\in{\mathscr P}_n : \vert \p(n)\vert\neq n\}.
$$
We remark that ${\mathscr P}_n^*$ contains the irreducible signed permutations. Thus we have the following inclusions
$${\mathscr P}_n^{\rm irred}\subset{\mathscr P}_n^{*}\subset {\mathscr P}_n.$$

In what follows, given $\p\in{\mathscr P}_n$ we let $\pi=\vert\p\vert=(\vert p_1\vert,\ldots,\vert p_n\vert)$ and $\theta=(p_1/\vert p_1\vert,\ldots,p_n/\vert p_n\vert)$.





\subsection{Rauzy induction}

The \emph{Rauzy induction}
is the operator $\mathcal{R}$ on the space of   $n$-IETs that associates to
$T=T_{(\lambda,\p)}$ with $\p \in {\mathscr P}_n^*$ the IET \mbox{$T_{(\lambda',\p')}=\mathcal{R}(T_{(\lambda,\p)})$} which is the
 first return map induced by $T$ on the
subinterval $I(\lambda,\p)=[0,\xi]$, where
\begin{eqnarray*}
\xi=\left\{
\begin{array}{ll}
\Vert \lambda \Vert-\lambda_{ \pi^{-1}(n)}, &\text{if}\ \lambda_{ \pi^{-1}(n)}<\lambda_n\\
\Vert \lambda \Vert-\lambda_n,              &\text{if}\ \lambda_{ \pi^{-1}(n)}>\lambda_n\\
\end{array}
\right..
\end{eqnarray*}
The Rauzy operator $T_{(\lambda,\p)}\mapsto T_{(\lambda',\p')}$ induces the map
$(\lambda,\p)\in\Lambda_n\times{\mathscr P}_n^*\mapsto (\lambda',\p')\in \Lambda_n\times {\mathscr P}_n$ in the data space, which we will denote by  $\overline{\mathcal{R}}$.
The domains of the maps $\overline{\mathcal{R}}:\Lambda_n\times
{\mathscr P}_n^*\to \Lambda_n\times {\mathscr P}_n$ and $\mathcal{R}$ are, respectively, the open full measure sets
$${\rm Dom}\,(\overline{\mathcal{R}})=\left\{ (\lambda,\p)\in\Lambda_n\times{\mathscr P}_n^*\mid\lambda_{\pi^{-1}(n)}\neq\lambda_n \right\},\quad {\rm Dom}(\mathcal{R})=\left\{T_{(\lambda,\p)}\mid (\lambda,\p)\in {\rm Dom}\,(\overline{\mathcal{R}})\right\}.$$
The map $\overline{\mathcal{R}}$ may be written in the form
\begin{eqnarray*}
\overline{\mathcal{R}}(\lambda,\p)=\left \{
\begin{array}{lll}
\Big(M_a(\p)^{-1}\lambda,a(\p)\Big)&\text{if}&\  \lambda_{ \pi^{-1}(n)}<\lambda_n\\
\Big(M_b(\p)^{-1}\lambda,b(\p)\Big)&\text{if}&\  \lambda_{ \pi^{-1}(n)}>\lambda_n\\
\end{array}
\right.,
\end{eqnarray*}
where the matrices $M_a(\p),M_b(\p)$ and the Rauzy maps $a,b:{\mathscr P}_n^*\to
{\mathscr P}_n$ are described below. For $i,j=1,\ldots,n$, denote
by $E_{ij}$ the $n\times n$ matrix of which the $(i,j)$th entry is
equal to $1$, and all other entries are  $0$. Let $I$
be the $n\times n$ identity matrix. The matrices $M_a$ and $M_b$ are
defined by
\begin{eqnarray*}
M_a(\p)&=&I+E_{n, \pi^{-1}(n)},\\
M_b(\p)&=&\sum_{i=1}^{ \pi^{-1}(n)}
E_{i,i}+E_{n,s(\p)}+\sum_{i= \pi^{-1}(n)}^{n-1} E_{i,i+1},
\end{eqnarray*}
where $s(\p)= \pi^{-1}(n)+(1+\theta_{ \pi^{-1}(n)})/2$.

We now define the Rauzy maps. When $\theta_n=+1$ the Rauzy map
$a: {\mathscr P}_n^* \to {\mathscr P}_n$ is defined by
$$a(\p)_i=\left\{
\begin{array}{ll}
\theta_i \pi_i, & \pi_i\leq \pi_n, \\
\theta_i (\pi_n+1), & \pi_i=n, \\
\theta_i (\pi_i+1), & \textrm{otherwise},\\
\end{array}
\right.
$$
and when $\theta_n=-1$, we have
$$
a(\p)_i=\left\{
\begin{array}{ll}
\theta_i \pi_i, & \pi_i\leq \pi_n-1, \\
-\theta_i \pi_n, & \pi_i=n, \\
\theta_i (\pi_i+1), & \textrm{otherwise}.\ \end{array} \right.
$$
The Rauzy map $b: {\mathscr P}_n^* \to {\mathscr P}_n$ when $\theta_{\pi^{-1}(n)}=+1$ is defined by
$$
b(\p)_i=\left\{
\begin{array}{ll}
\theta_i \pi_i, & i\leq  \pi^{-1}(n), \\
\theta_n \pi_n, & i=  \pi^{-1}(n)+1, \\
\theta_{i-1} \pi_{i-1}, & \textrm{otherwise},\\
\end{array}
\right.
$$
and when $\theta_{\pi^{-1}(n)}=-1$ by
$$
b(\p)_i=\left\{
\begin{array}{ll}
\theta_i \pi_i, & i\leq  \pi^{-1}(n)-1, \\
-\theta_n\pi_n, & i=  \pi^{-1}(n), \\
\theta_{i-1} \pi_{i-1}, & \textrm{otherwise}.\\
\end{array}
\right.
$$

\section{Stability of invariant components}


Given $(\lambda,\p) \in \Lambda_n \times{\mathscr P}_n^{*}$, we associate to the $n$-IET
$T=T_{(\lambda,\p)}$ the sequence
$T^{(0)},T^{(1)},T^{(2)},\ldots$ of $n$-IETs defined recursively by $T^{(0)}=T$ and
$T^{(k)}=T_{(\lambda^{(k)},\p^{(k)})}=\mathcal{R}(T^{(k-1)})=\mathcal{R}^k(T^{(0)})$ for all integers $k\ge 1$ such that $T^{(k-1)}\in {\rm Dom}\,(\mathcal{R})$. 
We say that $T$ has {\it finite expansion} if there exists $\ell=\ell(T)\ge 0$ such that 
$T\in {\rm Dom}\,(\mathcal{R}^{\ell})$ and $T^{(\ell)}$ is a reducible $n$-IET.
The number $\ell(T)$ is the least positive integer $m\ge 0$ such that $\p^{(m)}\in{\mathscr P}_n^{\rm red}$.



\begin{lemma}\label{bound} Let $(\lambda,\p)\in \Lambda_n^*\times{\mathscr P}_n$ be such that $\theta_i=p_i/\vert p_i\vert=-1$ and $\lambda_i>\Vert\lambda\Vert/2$ for some $i\in\{1,\ldots,n\}$. Then $\p$ is a reducible permutation.
\end{lemma}
\begin{proof} Suppose that $(\lambda,\p)$ satisfies the hypotheses of the lemma. Let
$T=T_{(\lambda,\p)}$ and let ${\rm Dom}\,(T)=\bigcup_{j=1}^n (x_{j-1},x_j)$. It is easy to see
that the interval $(x_{i-1},x_i)$ contains a flipped fixed point $x^*$ of $T$. Let $\lambda_i'=x^*-x_{i-1}$
and $\lambda_i''=x_{i}-x^*$. Thus $\lambda_i=\lambda_i'+\lambda_i''$. Because $\lambda$ is irrational we have that either $\lambda'=(\lambda_1,\ldots,\lambda_{i-1},\lambda_i',\lambda_{i+1},\ldots,
\lambda_n)$ is irrational or $\lambda''=(\lambda_1,\ldots,\lambda_{i-1},\lambda_i'',\lambda_{i+1},\ldots,
\lambda_n)$ is irrational. Without loss of generality we assume that $\lambda''\in\Lambda_n^*$. Notice that $T$ reflects the interval $(x_{i-1},x_i)$ around the fixed point $x^*$. Hence, there exist
$s\in\{1,\ldots,n\}$ and $s$ integers $j_1,\ldots,j_s$, all of them distinct from $i$, such that
$\lambda_1+\ldots+\lambda_{i-1}=\lambda_{j_1}+\ldots+\lambda_{j_s}+\lambda_i''$. This means
that $\lambda''$ is rational, which is a contradiction.
\end{proof}

\begin{theorem}[Nogueira \cite{N2}]\label{lfu} Almost all irreducible interval exchange transformation with flip(s) have finite expansion and so the function $\ell:\Lambda_n\times{\mathscr{P}}_n^{\rm irred}\to\N$ is defined almost everywhere. 
\end{theorem}
\begin{proof} Let $\p\in{\mathscr P}_n^{\rm irred}$ be a permutation with flips.
By the proof of Corollary 3.3 of \cite[p. 521]{N2}, there exists a full measure set $B_n\subset\Lambda_n^*$ such that if $\lambda\in B_n$ then either (a) $T_{(\lambda,\p)}$ has finite expansion or (b) $\theta_i=p_i/\vert p_i\vert=-1$ and $\lambda_i>\Vert\lambda\Vert/2$ for some $i\in\{1,\ldots,n\}$. In case (b), by Lemma \ref{bound}, $T_{(\lambda,\p)}$ has finite expansion.
\end{proof}



Given $(\lambda,\p)\in\Lambda_n\times\mathscr{P}_n^{*}$ and $k\ \ge 1$,  let $\nu_k:\Lambda_n\times{\mathscr P}_n^{*} \to\Lambda_n$ and ${\bf r}_k:\Lambda_n\times{\mathscr P}_n^{*} \to{\mathscr P}_n$ be the maps defined by
\begin{equation}\label{vr}
(\nu_k(\lambda,\p),{\bf r}_k(\lambda,\p))=\overline{\mathcal{R}}^{k}(\lambda,\p). 
\end{equation}

\begin{proposition}\label{abr} The domain of the maps $\nu_k$ and ${\bf r}_k$ is an open subset $U_k\subset\Lambda_n\times {\mathscr P}_n^{*}$. Furthermore, for each $(\mu,\p)\in U_k$ there exists
a neighborhood $W\times \{\p\}\subset U_k$ of $(\mu,\p)$ such that:
\begin{itemize}
\item [(a)] The map $\lambda\in W\mapsto {\bf r}_k(\lambda,\p)$ is constant; 
\item [(b)] The map $\lambda\in W\mapsto \nu_k(\lambda,\p)$ is a restriction of a linear isomorphism.
\end{itemize}
\end{proposition}
\begin{proof} We proceed by induction on $k$.
By definition, ${\rm Dom}\,({\nu}_k)={\rm Dom}\,({\bf r}_k)={\rm Dom}\,(\overline{\mathcal{R}}^{k})$. For $k=1$,  notice that ${\rm Dom}\,(\overline{\mathcal {R}})$ is an open subset of $\Lambda_n\times {\mathscr P}_n^{*}$. Let ${(\mu,\p)}\in {\rm Dom}\,(\overline{\mathcal{R}})$. By the definition of $\overline{\mathcal{R}}$, there exist $c\in \{a,b\}$
and an open neighborhood $W\subset\Lambda_n$ of $\mu$ such that $W\times\{\p\}\subset {\rm Dom}\,(\overline{\mathcal{R}})$, ${\bf r}_1(\lambda,\p)=c(\p)$  and $\nu_1(\lambda,\p)=M_c (\p)^{-1}\lambda$ for all $\lambda\in W$. As $\vert{\rm det}\,(M_c(\p))\vert= 1$, we have that $\lambda\mapsto\mathcal{\nu}_1(\lambda,\p)$ is restriction of a linear isomorphism. Hence, Proposition \ref{abr} holds for $k=1$. Suppose now that the Proposition \ref{abr} holds for $k-1$. Let us prove that it then holds for $k$. We have that ${\rm Dom}\,(\overline{\mathcal{R}}^k)=\{(\mu,\p)\in {\rm Dom}\,(\overline{\mathcal{R}})\mid\overline{\mathcal{R}}(\mu,\p)\in {\rm Dom}\,(\overline{\mathcal{R}}^{k-1})\}$. Let $(\mu,\p)\in {\rm Dom}\,(\overline{\mathcal{R}}^{k})$. Hence
$(\alpha,{\bf q}):=\overline{\mathcal{R}}(\mu,\p)\in {\rm Dom}\,(\overline{\mathcal{R}}^{k-1})$. By the induction hypothesis, ${\rm Dom}\,(\overline{\mathcal{R}}^{k-1})$ is open. By the above, as
$\mathcal{R}(\mu,\p)=(\nu_1(\mu,\p),{\bf r}_1(\mu,\p))$, there exist a neighborhood $W\subset\Lambda_n$ of $\mu$ and a linear isomorphism $L:\R^n\to\R^n$ such that
$\overline{\mathcal{R}}(W\times\{\p\})\subset L(W)\times \{{\bf q}\}\subset {\rm Dom}\,(\overline{\mathcal{R}}^{k-1})$. This means that
${\rm Dom}\,(\overline{\mathcal{R}}^k)$ is open. Properties (a) and (b) follow from the induction hypothesis and the identities ${\nu}_k(\lambda,\p)=\nu_1\big( \nu_{k-1}(\lambda,\p), {\bf r}_{k-1}(\lambda,\p)  \big)$ and ${\bf r}_k(\lambda,\p)={\bf r}_1\big( \nu_{k-1}(\lambda,\p), {\bf r}_{k-1}(\lambda,\p)  \big)$.
\end{proof}


\begin{corollary}\label{same} Let $T\in{\rm Dom}\,(\mathcal{R}^k)$ and $T^{(k)}=\mathcal{R}^k(T)$.
Then there is a bijection between the invariant components of $T$ and those of $T^{(k)}$. More precisely, 
each periodic component $($respectively stable periodic component, minimal component, robust minimal component$)$ of 
$T$ is associated to one and only one periodic component $($respectively stable periodic component, minimal component, robust 
minimal component$)$ of $T^{(k)}$.
\end{corollary}

\begin{proof} The bijection of periodic and transitive component of $T$ and $T^{(k)}$ holds since 
 $T^{(k)}$ is the first return map of $T:(0,c)\to (0,c)$ to some subinterval $(0,\rho) \subset (0,c)$, and since $T (\rho,c)
\subset (0,\rho)$. By \mbox{Proposition \ref{abr}}, the map $\lambda\mapsto\nu_k(\lambda,\p)$ takes sets of positive measure 
to sets of positive measure. Since $\nu_k(\lambda,\p)$ gives the length vector of $T^{(k)}$, stable periodic components 
(respectively robust minimal components) of $T$ correspond to stable periodic components 
(respectively robust minimal components) of $T^{(k)}$.
\end{proof}

We say that ${\bf q}=(q_1,\ldots,q_n)\in{\mathscr P}_n^{\rm red}$ is {\it reducible} into $s\in\{2,\ldots,n\}$ 
irreducible signed permutations if there exist $s$ positive integers 
 $n_1$, $n_2$,\ldots, $n_s$ with $\sum_{j=1}^s n_j=n$ such that 
 $\{\vert {\bf q} ({d_j+1}) \vert,\ldots,\vert {\bf q}(d_{j}+n_j) \vert \}=\{d_{j}+1,\ldots,d_{j}+n_j \}$ for all $j\in\{1,\ldots,s\}$, where
 $n_0=0$ and $d_j=\sum_{k=0}^{j-1} n_k$. In this case, ${\bf q}$ induces $s$ irreducible permutations
 ${\bf v}_1\in{\mathscr P}_{n_1}^{\rm irred}$,\ldots, ${\bf v_s}\in{\mathscr P}_{n_s}^{\rm irred}$ defined by
 $$
{\bf v}_j(i)=\dfrac{{\bf q}(i+d_j)}{\vert{\bf q}(i+d_j)\vert}\big(\vert {\bf q}(i+d_j)\vert-d_j  \big),\quad i\in\{1,\ldots, n_j\}.
$$
We will write
 \begin{equation}\label{decompo}
 {\bf q}=({\bf v}_1,\ldots,{\bf v}_s)\in {\mathscr P}_{n_1}^{\rm irred}\times\cdots\times{\mathscr P}_{n_s}^{\rm irred}
 \end{equation}
 to mean that ${\bf q}$ is reducible into the irreducible permutations ${\bf v}_1$,\ldots, ${\bf v}_s$.
 This decomposition is clearly unique.
 Thus if ${\bf q}$ is reducible, then for any length vector $\lambda$, 
the domain of the $n$-IET $T_{(\lambda,{\bf p})}$ decomposes into
 $s$ invariant sets where $s$ is given by Equation \eqref{decompo}.
 The restriction of $T_{(\lambda,{\bf p})}$ to the $j$th-invariant set is an irreducbile $n_j$-IET with permutation vector ${\bf v}_j$ and length vector $(\lambda_{h_j}, \dots, \lambda_{h_j + n_j-1})$
where $h_j = 1 + \sum_{k<j}  n_k$ for $j \ge 1$.\\

In the proof that follows below, we will use the following notation: given
 $U\subset \Lambda_n\times {\mathscr P}_n$ we let $\mathcal{T}(U)=\{T_{(\lambda,\p)}\mid
(\lambda,\p)\in U\}$.\\

\noindent{\it Proof of Theorem B}.
Keane has shown that Theorem B holds for oriented IETs \cite{K}.
In  fact, almost every oriented IET has one robust minimal component and no periodic component.
Now we consider the case of IETs with flip(s). 
By decomposing the IET into irreducible interval exchanges we may assume that
the initial IET is irreducible. The proof proceeds by induction on the number of intervals.  Of course, every IET with flip of 1 interval 
has only 1 stable periodic component and no minimal component. Suppose as the induction hypothesis that Theorem B holds 
for all IETs of $k\le n-1$ intervals. More precisely, suppose that for $1 \le k \le n-1$, 
there is a set of full measure $C_k \subset \Lambda_k$ such that for any $k$-IET in
$\mathcal{T}(C_k\times{\mathscr P}_k)$
 every periodic component is stable and 
every minimal component is robust. By \mbox{Theorem \ref{lfu}}, there exists a full measure set $B_n\subset\Lambda_n^*$ 
such that all $n$-IET in $\mathcal{T}(B_n\times{\mathscr P}_n)$ with flip(s) have finite expansion. Let 
$T = T_{(\lambda,\p)}\in \mathcal{T}(B_n\times {\mathscr P}_n^{\rm irred})$ be an $n$-IET with flip(s). Then there exists $m=\ell(\lambda,\p)\in\N$ 
such that $T\in {\rm Dom}\,({\mathcal R}^m)$ and $T^{(m)}=\mathcal{R}^m(T)$ is a reducible $n$-IET. 

Let $T_{(\boldsymbol{\mu},{\bf q})} := T^{(m)}$. By definition, ${\bf q}\in {\mathscr P}_n^{\rm red}$
and since $\lambda \in \Lambda_n^*$, we have
$\boldsymbol{\mu}\in\Lambda_n^*$.
 Let ${\bf q}=({\bf v}_1,\ldots,{\bf v}_n)\in{\mathscr P}_{n_1}^{\rm irred}\times\cdots\times{\mathscr P}_{n_s}^{\rm irred} $ be the decomposition of ${\bf q}$ into irreducible signed permutations
 and let
${\boldsymbol{\mu}}=(\boldsymbol{\mu_1},\ldots,\boldsymbol{\mu}_s)\in\Lambda_{n_1}^*\times\cdots\times\Lambda_{n_s}^*$be the associated decomposition of $\boldsymbol{\mu}$. 
By Proposition \ref{abr}, there exists a full measure set $C_n\subset B_n$ such that if ${(\lambda,\p)}\in C_n\times {\mathscr P}_{n}^{\rm irred}$  then $(\boldsymbol{\mu}_j,{\bf v}_j)\in C_{n_j}\times {\mathscr P}_{n_j}^{\rm irred}$ for all $1\le j\le s$. Hence, by the induction hypothesis, each $T_{(\boldsymbol{\mu}_j,{\bf v}_j)}$ has only stable periodic components and robust minimal components.  By Corollary \ref{same}, for each IET in $\mathcal{T}(
C_n\times {\mathscr P}_n)$ every periodic component is stable and every minimal component is robust.\cqd
 
 \section{Existence result}
 
 
 
 
 
 The aim of this section is proving Theorem C. Firstly we will introduce some notation and preparatory lemmas.
 
 Let $a,b:{\mathscr P}_n^*\to {\mathscr P}_n$ be the Rauzy maps and let ${F}:
 {\mathscr P}_n\to 2^{{\mathscr P}_n}$ be the set-valued function defined by $F(\p)=\{a(\p),b(\p)\}$ if $\p\in {\mathscr P}_n^*$ and
  $F(\p)=\emptyset$ if $\p\in {\mathscr P}_n\setminus {\mathscr P}_n^*$. We let $F:2^{{\mathscr P}_n}\to 2^{{\mathscr P}_n}$ be the induced map defined by $F(S)=\bigcup_{\p\in S} F(\p)$ for $S\subset\mathscr{P}_n$. The {\it forward set} of $\p\in {\mathscr P}_n^*$ is the set of permutations defined by
 $ {\mathscr F}(\p)=\bigcup_{k\ge 0}F^k(\{\p\})$. Notice that the forward set of a permutation $\p$ is the set of all permutations that can be obtained from $\p$ through applications of the Rauzy maps finitely many times.
  
 \begin{lemma}\label{lemimp} Let $\p\in {\mathscr P}_n^*$ and ${\mathbf q}\in {\mathscr F}(\p)$. There exist a positive measure set $B_n\subset \Lambda_n$, a positive integer $K\in\N$ and a linear isomorphism $L:\R^n\to \R^n$
 such that $(\lambda,\p)\in {\rm Dom}\,(\overline{\mathcal{R}}^K)$ and
 $\overline{\mathcal{R}}^K(\lambda,\p)=(L(\lambda),{\mathbf q})$ for all $\lambda\in B_n$.
 \end{lemma}
 \begin{proof} Since ${\mathbf q}\in\mathscr{F}(\p)$, there exists a sequence $\{(c_k,{\mathbf q}^{(k)})\}_{k=0}^K\subset\{a,b\}\times {\mathscr P}_n$ such that ${\mathbf q}^{(0)}=\p$, ${\mathbf q}^{(K)}={\mathbf q}$,
 $\{(c_k,{\mathbf q}^{(k)})\}_{k=0}^{K-1}\subset\{a,b\}\times{\mathscr P}_n^*$ and ${\mathbf q}^{(k)}=c_{k-1}({\mathbf q}^{(k-1)})$ for all $1\le k\le K$. Let
 $$M=M_{c_0}({\mathbf q}^{(0)})M_{c_1}({\mathbf q}^{(1)})\cdots M_{c_{(K-1)}}({\mathbf q}^{(K-1)})
 $$
 and let $B_n=M\Lambda_n$. Then every $T = T_{(\lambda,\p)}\in \mathcal{T}(B_n\times {\mathscr P}_n^*)$ has the property
 that $T\in {\rm Dom}\,(\mathcal{R}^K)$ and $T^{(K)}=\mathcal{R}^K(T)=T_{(L(\lambda),{\mathbf q})}$, where
 $L(\lambda)=M^{-1}\lambda$.
 \end{proof}

Now we introduce a concrete
 example of an interval exchange transformations of $7$ intervals having 2 robust minimal components and 3 stable periodic components. This example will be generalized later to prove Theorem C.
 
 Let $\p=(-7,6,5,-3,-4,-1,-2)\in {\mathscr P}_7^{\rm irred}$. We claim that ${\bf q}=(2,1,4,3,-5,-6,-7)\in{\mathscr F}(\bf{p})$. In fact, we have that ${\bf q}=b^6(\bf{p})$. By Lemma \ref{lemimp}, there exist irrational length vectors $\lambda,\mu\in \Lambda_7^*$ such that if $T_{(\lambda,\p)}$ then
 $T\in {\rm Dom}\,({\mathcal{R}}^6)$ and $T_{(\mu,{\bf q})}=T^{(6)}={\mathcal{R}}^6(T_{(\lambda,\p)})$. It is clear that $T^{(6)}$ can be decomposed into two irrational rotations and three interval exchanges of $1$ interval with 1 flip. Therefore, by Corollary \ref{same}, $T$ has two robust minimal components and three stable periodic components (see Figure \ref{Figura}).
 
In Figure 1, the orbit of five sample points have been plotted
on the $x$-axis, one in each invariant component of $T$. It is easy to see that the blue and the red orbits belong to minimal components whereas the black, green and orange orbits belong to periodic components. Notice that, as prescribed by Lemmas \ref{3.2} and \ref{3.3}, each invariant component has at least one singular point at the boundary, whereas the minimal components have additionally singular points in the interior of the component. Moreover, it is possible to identify the period of the periodic components, 
there are two periodic components of period 4 and 1 periodic component of \mbox{period 2}.\\

\noindent{\it Proof of Theorem C.} Firstly let $k=n$ and $\ell=0$ and consider the permutation
$$\p=(-n,n-1,n-2\ldots,2,1)\in{\mathscr P}_n^{\rm irred}.$$
Clearly  ${\bf q}=b^{n-1}(\p)=(-1,-2,\ldots,-(n-1),-n)$. Thus ${\bf q}\in {\mathscr F}(\p)$. By Lemma \ref{lemimp} there exist length vectors $\lambda,\mu\in \Lambda_n^*$ such that
$(\mu,{\bf q})=\overline{\mathcal{R}}^{n-1}(\lambda,\p)$. Let $T=T_{(\lambda,\p)}$.
Because $T^{(n-1)} = T_{(\mu,{\bf q})}$ has $n$ stable periodic components and no minimal component, we have by Corollary \ref{same} that $T_{(\lambda,{\bf q})}$ has $n$ stable periodic components and no minimal component.

Now let us consider the case in which $k\ge 2$, $1\le \ell <n/2$ and $k+2\ell\le n$. Let $\p\in{\mathscr P}_n^{\rm irred}$ be the following permutation
$$\mbox{\fontsize{10}{12}\selectfont $\p=(-n,n-1,\ldots,n-(k-1),-[n-(k-3)],-[n-(k-2)],\ldots,-(r+1),-(r+2),-1,-2,\ldots,-r),$}$$
where $r=n-k-2(\ell-1)=n-(k+2\ell)+2$. The number $r$ is the number of intervals necessary to form the $\ell$-th minimal component. We will construct an irreducible $n$-IET with flip(s) which has a Rauzy induced
 with $k$ flipped periodic components, $\ell-1$ robust minimal components of rotation type each, and one last minimal component consisting of an oriented $r$-IET. We have that:
$$\mbox{\fontsize{10}{12}\selectfont ${\bf q}=b^{n-1}(\p)=(\underbrace{r,\ldots,2,1}_{\rm min.},\underbrace{r+2,r+1}_{\rm min.},\ldots,
\underbrace{n-(k-2),n-(k-3)}_{\rm min.},\underbrace{-[n-(k-1)]}_{\rm per.},\ldots,\underbrace{-(n-1)}_{\rm per.},\underbrace{-n}_{\rm per.}).$}$$
Thus ${\bf q}\in {\mathscr F}(\p)$. By Lemma \ref{lemimp} there exist length vectors $\lambda,\mu\in \Lambda_n^*$ such that
$(\mu,{\bf q})=\overline{\mathcal{R}}^{n-1}(\lambda,\p)$. Let $T=T_{(\lambda,\p)}$. It is easy to see that $T^{(n-1)} = T_{(\mu,{\bf q})}$ has $k$ stable (flipped) periodic components and $\ell$ robust minimal components. By Corollary \ref{same}, $T_{(\lambda,{\bf p})}$ has $k$ stable periodic components and $\ell$ robust minimal components. 

Finally, in the case $k=1$, $1\le \ell< n/2$ and $k+2\ell\le n$, let $\p\in{\mathscr P}_n^{\rm irred}$ be the following permutation:
$$\mbox{\fontsize{10}{12}\selectfont $\p=(-n,-(n-2),-(n-1),\ldots,-(r+1),-(r+2),-1,-2,\ldots,-r),$}$$
where $r=n-2\ell+1$. In this case, 
$$\mbox{\fontsize{10}{12}\selectfont ${\bf q}=b^{n-1}(\p)=(\underbrace{r,\ldots,2,1}_{\rm min.},\underbrace{r+2,r+1}_{\rm min.},\ldots,\underbrace{n-1,n-2}_{\rm min.},\underbrace{-n}_{\rm per.}).$}$$
Thus ${\bf q}\in {\mathscr F}(\p)$ and we may apply the same reasoning as above. This proves the $(\Leftarrow)$ part of Theorem C. 

Now let us prove the $(\Rightarrow)$ part. Let ${\mathscr P}_n^f$ be the
subset of ${\mathscr P}_n^{\rm irred}$ formed by the irreducible permutations having flip(s).
A combination of Theorem A, \mbox{Theorem B}, Theorem \ref{lfu} and Nogueira's \mbox{result \cite{N2}} implies that there exists a subset \mbox{$B_n\subset\Lambda_n^*$} of full measure such that if $T\in\mathcal{T}(B_n\times {\mathscr P}_n^f)$ then
$T$ has $n_{\rm per}$ (stable) periodic components, $n_{\rm min}$ (robust) minimal components and $T$ has finite expansion. Moreover, $n_{\rm per}\ge 1$, $0\le n_{\rm min}< n/2$ and $n_{\rm per}+2\,n_{\rm min}\le n$. We claim that for almost all $T\in\mathcal{T}(B_n\times {\mathscr P}_n^f)$,  $n_{\rm min}=0$ implies $n_{\rm per}=n$. We will prove this by induction \mbox{on $n$}. Suppose that for all $1\le k \le n-1$, there exists a full measure set
$C_k\subset\Lambda_k$ such that all $k$-IET in $\mathcal{T}(C_k\times {\mathscr P}_k^f)$ without minimal components have $k$ stable periodic components. By the definition of $B_n$, for all $T\in \mathcal{T}(B_n\times {\mathscr P}_n^f)$, there exists $m=m(T)\in \N$ such that $T\in {\rm Dom}\,(\mathcal{R}^m)$
 and $T^{(m)}=\mathcal{R}^m(T)\in\mathcal{T}(\Lambda_n\times{\mathscr P}_n^{\rm red})$. 

Given ${(\lambda,\p)}\in B_n\times {\mathscr P}_n^f$, let $T=T_{(\lambda,\p)}$ and $T_{(\boldsymbol{\mu},{\bf q})} = T^{(m)}$.  By definition, ${\bf q}\in {\mathscr P}_n^{\rm red}$
and since $\lambda \in \Lambda_n^*$, we have
$\boldsymbol{\mu}\in\Lambda_n^*$. Let ${\bf q}=({\bf v}_1,\ldots,{\bf v}_n)\in{\mathscr P}_{n_1}^{\rm irred}\times\cdots\times{\mathscr P}_{n_s}^{\rm irred} $ be the decomposition of ${\bf q}$ into irreducible signed permutations
 and let
${\boldsymbol{\mu}}=(\boldsymbol{\mu_1},\ldots,\boldsymbol{\mu}_s)\in\Lambda_{n_1}^*\times\cdots\times\Lambda_{n_s}^*$be the associated decomposition of $\boldsymbol{\mu}$. Notice that $n_1+\ldots+n_j=n$.
 By Proposition \ref{abr}, there exists a full measure set $C_n\subset B_n$ such that if ${(\lambda,\p)}\in C_n\times {\mathscr P}_{n}^{\rm irred}$  then $(\boldsymbol{\mu}_j,{\bf v}_j)\in C_{n_j}\times {\mathscr P}_{n_j}$ for all $1\le j\le s$. Now let $(\lambda,\p)\in C_n$.
By Corollary \ref{same}, since $T=T_{(\lambda,\p)}$ has no minimal component we have that
$T_{(\boldsymbol{\mu},{\bf q})}$ have no minimal component. 
Consequently, $T_{({\boldsymbol \mu}_j,{\bf q}_j)}$ have no minimal component for all $1\le j\le s$.
Then by Keane \cite{K}, each $T_{({\boldsymbol \mu}_j,{\bf q}_j)}$ is either an oriented periodic 1-IET or
an $n_j$-IET with flips. In the second case, the induction hypothesis implies that  $T_{({\boldsymbol \mu}_j,{\bf q}_j)}$ has $n_j$ stable periodic components.
By Corollary \ref{same} and by the above, each IET in $\mathcal{T}(
C_n\times {\mathscr P}_n^{\rm irred})$ with flips without minimal components has $n_1+\ldots + n_j=n$ stable periodic components.
\cqd\\

\begin{figure}[h!]
\caption{Interval exchange transformation of 7 intervals}
 \centering
  \includegraphics[width=10cm]{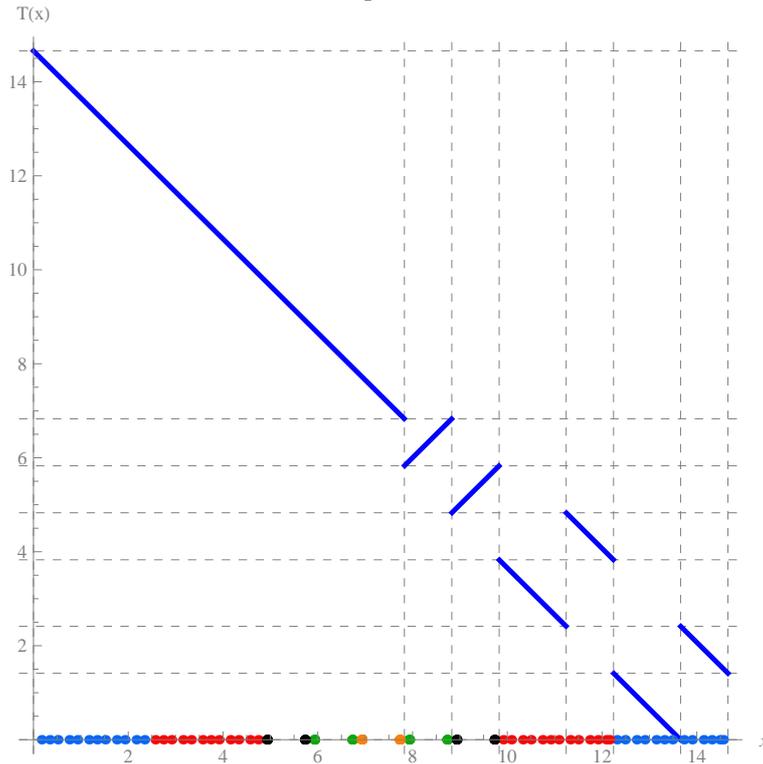}\label{Figura}
  \end{figure}

\noindent{\it Acknowledgments.} Part of this article was developed during the visit of the second author to the Institut de Math\'ematiques de Luminy (IML) - Universit\'e de la M\'editerran\'ee. B. Pires would like to acknowledge IML for the kind hospitality.

\end{document}